\documentclass[a4paper,12pt,reqno]{amsart}
\usepackage{enumerate,amsmath,amsthm,amssymb,cite,mathrsfs,eufrak,graphicx,caption,subfigure}
\DeclareMathOperator{\re}{\mathrm{Re}}
\DeclareMathOperator{\imag}{\mathrm{Im}}
\DeclareMathOperator{\se}{\mathscr{S}^*_e}
\DeclareMathOperator{\ke}{\mathscr{K}_e}
\DeclareMathOperator{\s}{\mathfrak{s}_{\mu,\nu}}
\DeclareMathOperator{\h}{\mathfrak{h}_{\mu,\nu}}
\DeclareMathOperator{\f}{\mathfrak{f}_{\mu,\nu}}

\numberwithin{equation}{section}
\newtheorem{theorem}{Theorem}[section]
\newtheorem{lemma}[theorem]{Lemma}
\newtheorem{corollary}[theorem]{Corollary}
\theoremstyle{remark}

\newtheorem{example}[theorem]{Example}

\makeatletter
\@namedef{subjclassname@2010}{%
	\textup{2010} Mathematics Subject Classification}
\makeatother

\begin{document}

\title[Exponential Starlikeness and Convexity]{Exponential Starlikeness and Convexity of Confluent Hypergeometric, Lommel and Struve Functions}
\author[Adiba Naz]{Adiba Naz}
\address{Department of Mathematics, University of Delhi, 	Delhi--110 007, India}
\email{adibanaz81@gmail.com}

\author[Sumit Nagpal]{Sumit Nagpal}
\address{Department of Mathematics, Ramanujan College, University of Delhi,	Delhi--110 019, India}
\email{sumitnagpal.du@gmail.com}

\author[V. Ravichandran]{V. Ravichandran}
\address{Department of Mathematics, National Institute of Technology,  Tiruchirappalli--620 015,  India}
\email{vravi68@gmail.com}

\begin{abstract}
	Sufficient conditions are obtained on the parameters of Lommel function of the first kind, generalized Struve function of the first kind and the confluent hypergeometric function under which these special functions become exponential convex and exponential starlike in the open unit disk.	The method of differential subordination is employed in proving the results. Few examples are also provided to illustrate the  results obtained.
\end{abstract}


\maketitle

\section{Introduction}
The remarkable use of generalized hypergeometric functions aroused great interest among researchers in the last few decades, specifically after  they were used by Louis de Branges \cite{MR772434}  in the proof of the Milen conjecture which in turn imply the famous Bieberbach conjecture. There is an extensive literature in geometric function theory that deals with the analytic and geometric properties of different kinds of  hypergeometric functions like Bessel function \cite{MR3568675,MR3753028,MR2656410,MR2429033,MR123456,MR2486953,MR2826152,MR2743533},  Struve function \cite{MR3035216,MR3718596,MR3528732,MR3927311,MR3596935}, confluent hypergeometric function \cite{MR3738359,MR1017006,MR1637348,MR3351038}, Lommel function \cite{MR3353311,sim2018geometric,MR3596935} and other generalized hypergeometric functions \cite{MR826655,MR918587}. Using numerous techniques, the authors determined various sufficient conditions on the parameters involved in these special functions to belong to the class of univalent functions or in its subclasses of functions that are convex, starlike, close-to-convex, uniformly convex and so forth. In this paper, several sufficient conditions are obtained on the parameters involved in Lommel function of the first kind, generalized Struve function of the first kind and the confluent hypergeometric function to belong to the class of normalized univalent functions $f$ defined in the unit disk $\mathbb{D}:=\{z\in\mathbb{C}\colon|z|<1\}$ such that $w=zf'/f$ or $w=1+zf''/f'$ lies in the domain $|\log w|<1$ of the right-half plane associated with the exponential function.

The concept of subordination plays a crucial role in defining several well-known classes of analytic functions defined in the unit disk $\mathbb{D}$. If $f$ and $g$ are two analytic functions in $\mathbb{D}$, then $f$ is subordinate to $g$, if there exists an analytic  function $w$ in $\mathbb{D}$ such that $w(0)=0$ and $|w(z)|<1$ for all  $z\in\mathbb{D}$ that satisfies  $f=g\circ w$ in $\mathbb{D}$. Symbolically, we write it as $f\prec g$. If  $f\prec g$, then $f(0)=g(0)$ and $f(\mathbb{D})\subseteq g(\mathbb{D})$. In particular, if $g$ is univalent in $\mathbb{D}$, then $f\prec g$ if and only if $f(0)=g(0)$ and $f(\mathbb{D})\subseteq g(\mathbb{D})$. The family $\mathscr{P}_e$ consists of analytic functions $p$  in $\mathbb{D}$ with $p(0)=1$  and $p(z)\prec e^z$ for all $z\in \mathbb{D}$.

Let $\mathscr{A}$ denote the class of all analytic functions $f$ normalized by the conditions $f(0)=f'(0)-1=0$ and its subclass consisting of all univalent functions is denoted by $\mathscr{S}$.  A function $f\in\mathscr{A}$ is said to be \emph{exponential convex} (or \emph{exponential starlike}) if the quantity $1+zf''(z)/f'(z)$ (or $zf'(z)/f(z)$) belongs to $\mathscr{P}_e$. We denote the classes of such functions by $\ke$ and $\se$ respectively. These classes are particular case of the subclasses of convex and starlike functions introduced by Ma and Minda \cite{MR1343506} and were extensively studied by Mendiratta \textit{et al.\@}  \cite{MR3394060}.

If $\Gamma$ is the Euler's gamma function and $(x)_\mu$ is the Pochhammer symbol given by
\begin{equation*} (x)_\mu:=\frac{\Gamma(x+\mu)}{\Gamma (x)}= \begin{cases}
1, & \mu=0\\
x(x+1)\cdots (x+n-1), &\mu=n\in\mathbb{N}
\end{cases}\end{equation*}
then the generalized hypergeometric function ${}_qF_s(\alpha_1, \ldots,\alpha_q;\beta_1,\ldots,\beta_s;z )$ is defined by
\[ {}_qF_s(\alpha_1, \ldots,\alpha_q;\beta_1,\ldots,\beta_s;z )=\sum_{n\geq0} \frac{(\alpha_1)_n\ldots(\alpha_q)_n}{(\beta_1)_n\ldots(\beta_s)_n}\frac{z^n}{n!} \]
where  $\alpha_i\in\mathbb{C}$  ($i=1,2,\ldots q$), $\beta_j\in\mathbb{C}\setminus\{0,-1,-2,\ldots\}$  $(j=1,2,\ldots s)$  and $q$, $s$ are  nonnegative integers such that $q\leq s+1$. This paper aims to determine the sufficient conditions for some generalized hypergeometric functions  to be in the classes $\ke$ and $\se$. Sections \ref{confluent}, \ref{lommel} and \ref{struve} are devoted to confluent hypergeometric function, Lommel function of the first kind and generalized Struve function of the first kind respectively. Few examples are also discussed in each section to illustrate the results obtained.

We shall make use of the theory of differential subordination (introduced by Miller and Mocanu \cite{MR1760285}) in proving our results. In \cite{MR3962536}, Naz \textit{et al.\@}  discussed the class of admissible functions associated with the exponential function $e^z$ and proved the following key lemma:
\begin{lemma}\cite{MR3962536}\label{lemA}
Let $\Omega$ be a subset of $\mathbb{C}$ and the function $\Psi \colon \mathbb{C}^3\times \mathbb{D}\to\mathbb{C}$ satisfies the admissibility condition $ \Psi(r,s,t;z) \notin \Omega$ whenever
 \begin{equation*}
 r = e^{e^{i\theta}},\quad  s = m e^{i\theta} r \quad
 \text{and} \quad \re (1+t/s) \geq m (1+\cos \theta)
 \end{equation*}
 where $z\in\mathbb{D}$, $\theta \in [0,2\pi)$ and $m \geq 1$. If $p$ is an analytic function in $\mathbb{D}$ with $p(0)=1$ and $\Psi(p(z), zp'(z), z^2 p''(z); z) \in \Omega$ for $z\in \mathbb{D}$, then   $p \in \mathscr{P}_e$.
 \end{lemma}
It is worth to note that the admissiblity condition $\Psi(r,s,t;z)\not\in\Omega$ is verified for all $r=e^{e^{i\theta}}$, $s=e^{e^{i\theta}}e^{e^{i\theta}}$ and $t$ with $\re(1+t/s)\geq 0$, that is,  $\re ((t+s)e^{-i\theta}e^{-e^{i\theta}}) \geq 0$ for all $\theta\in[0,2\pi)$ and $m\geq1$. This form of admissibility condition will be utilized throughout the paper. Moreover, for the case $\Psi \colon  \mathbb{C}^2\times \mathbb{D}\to\mathbb{C}$, the admissibility condition reduces to
\[\Psi(e^{e^{i\theta}}, m e^{i\theta} e^{e^{i\theta}};z) \notin \Omega \]
where $z\in\mathbb{D}$,  $\theta \in [0,2\pi)$ and $m \geq 1$.

\section{Confluent Hypergeometric Function}\label{confluent}
For complex constants $a$ and $c$ with $c\ne0$, $-1$, $-2$, $\ldots$, the \textit{confluent} (or \textit{Kummer}) \textit{hypergeometric function} $\Phi(a;c;z)$ defined by
\begin{align*}
\Phi(a;c;z) &={}_1 F_{1}(a;c;z) =1+\frac{a}{c}\frac{z}{1!}+\frac{a(a+1)}{c(c+1)}\frac{z^2}{2!}+\cdots\\
&=\sum_{n=0}^{\infty}\frac{(a)_n}{(c)_n}\frac{z^n}{n!}=\frac{ \Gamma(c)}{\Gamma(a)}\sum_{n=0}^{\infty}\frac{\Gamma(a+n)}{\Gamma(c+n)}\frac{z^n}{n!}
\end{align*}
is an entire function and it satisfies the second-order Kummer's differential equation
\begin{equation}\label{con1}
z\omega''(z)+(c-z)\omega'(z)-a\omega(z)=0 \qquad (z\in\mathbb{C}).
\end{equation}
Moreover, the function $\Phi(a;c;z)$ satisfies the following: \begin{equation}\label{elm}
a\Phi(a+1;c+1;z)=c\Phi'(a;c;z)\quad \mbox{and}\quad \Phi(a;a;z)=e^z.
\end{equation}
Miller and Mocanu \cite{MR1017006} investigated the univalence, convexity and starlikeness of confluent hypergeometric functions using the technique of differential subordination. Ponnusamy and Vuorinen \cite{MR1637348} carried out the similar analysis by using the sufficient conditions of close-to-convexity and starlikeness in terms of monotonicity of coefficients.  Ali \textit{et al.\@}  \cite{MR3351038} determined the Janowski convexity and Janowski starlikeness of $\Phi(a;c;z)$. Recently, Bohra and Ravichandran \cite{MR3738359} determined the conditions under which $\Phi(a;c;z)$ is strongly convex of order $1/2$. The first theorem of this section determines the conditions on parameters $a$ and $c$ so that $\Phi(a;c;z)$ belongs to the class $\mathscr{P}_e$.

\begin{theorem}\label{conf1}
Let the parameters $ a$, $c\in\mathbb{C}$  be constrained such that $c$ is not a nonnegative integer and $\re(c)\geq|a|+2$. Then $\Phi(a;c;z)\in\mathscr{P}_e$.
\end{theorem}
\begin{proof}
First, we claim that the function $(c/a)\Phi'(a;c;z)\in\mathscr{P}_e$ if $a\neq 0$ and $\re(c)\geq|a+1|+1$. Define a function $p\colon\mathbb{D}\to\mathbb{C}$ by \begin{equation*}
	p(z)=\frac{c}{a}\Phi'(a;c;z).
	\end{equation*}Clearly $p$ is analytic in $\mathbb{D}$ as $a\neq 0$. Since $\Phi(a;c;z)$ satisfies \eqref{con1}, the function $p$ satisfies the second-order differential equation
	\begin{equation*}
	z^2p''(z)+(c-z+1)zp'(z)-(a+1)zp(z)=0.
	\end{equation*} Let us define another function $\Psi\colon\mathbb{C}^3\times\mathbb{D}\to\mathbb{C}$ by \[\Psi(r,s,t;z)=t+s+(c-z)s-(a+1)rz \] and suppose that $\Omega:=\{0\}$. Then $\Psi(p(z),zp'(z),z^2p''(z);z)\in\Omega$ for all $z\in\mathbb{D}$. In order to prove that $p(z)\prec e^z$, we must show $\Psi(r,s,t;z)\notin\Omega$ whenever $r = e^{e^{i\theta}}$, $s = m e^{i\theta} e^{e^{i\theta}}$, $\re ((s+t)e^{-i\theta}e^{-e^{i\theta}}) \geq 0 $, $\theta\in[0,2\pi)$, $z\in \mathbb{D}$ and $m \geq 1$, in view of Lemma \ref{lemA}. Recall the  inequality $|z_1\pm z_2|\geq ||z_1|-|z_2||$ for all $z_1$, $z_2\in\mathbb{C}$ and consider
	\begin{align*}
	|\Psi(r,s,t;z)|& = |e^{e^{i\theta}}|\cdot|(t+s)e^{-e^{i\theta}}+(c-z)me^{i\theta}-(a+1)z|\\& \geq e^{\cos\theta}\left[|(t+s)e^{-i\theta}e^{-e^{i\theta}}+(c-z)m|-|a+1|\cdot|z|\right]\\
	& \geq \frac{1}{e}\left[ \re((t+s)e^{-i\theta}e^{-e^{i\theta}})+\re(c-z)m -|a+1|\cdot|z|\right]\\
&\geq \frac{1}{e}\left[ \re(c-z)m -|a+1|\cdot|z|\right].
\end{align*}
Since $\re(z)<1$ and $\re(c)\geq|a+1|+1\geq1$, it follows that $\re(c-z)>\re(c-1)\geq 0$ so that
	\begin{align*}
	|\Psi(r,s,t;z)|\geq\frac{1}{e}\left[ \re(c-z) -|a+1|\right]>\frac{1}{e}\left[ \re(c)-1-|a+1|\right]\geq0.
	\end{align*}This proves that $|\Psi(r,s,t;z)|> 0$ and therefore $p(z)\prec e^z$ for all $z\in\mathbb{D}$. Hence we have shown that $(c/a)\Phi'(a;c;z)\in\mathscr{P}_e$ if $a\neq 0$ and $\re(c)\geq|a+1|+1$. Using the relation \eqref{elm}, we can write
\[(a-1)\Phi(a;c;z)=(c-1)\Phi'(a-1;c-1;z).\]
Consequently, it follows that $\Phi(a;c;z)\in\mathscr{P}_e$ provided $a\ne 1$ and $\re(c)\geq|a|+2$. If $a=1$, then we will employ the limit procedure. Let $(a_n)$ be a sequence of complex numbers such that $\lim(a_n)=1$, $a_n\neq1$ and $\re(c)\geq|a_n|+2$ for all $n\in \mathbb{N}$. For all $z\in \mathbb{D}$, note that
 \[\Phi(a_n;c;z)=\frac{c-1}{a_n-1}\Phi'(a_n-1;c-1;z)\in \mathscr{P}_e\quad \mbox{for all }n\in \mathbb{N}\]
 so that
 \[|\log (\Phi(a_n;c;z))|<1\quad \mbox{for all }n\in \mathbb{N}.\]
 By the continuity of $\Phi$, it follows that $\Phi(1;c;z)\in\mathscr{P}_e$ if  $\re(c)\geq3$ (on letting $n\to\infty$).
\end{proof}
Let us illustrate Theorem \ref{conf1} through the following example for different choices of $a$ and $c$ satisfying the hypothesis.

\begin{example}
By Theorem \ref{conf1}, it is evident that the polynomials
\begin{gather*}q_1(z):=\Phi(-1;3;z)=1-\frac{z}{3}\\
q_2(z):=\Phi(-2;4;z)=1-\frac{z}{2}+\frac{z^2}{20}\\
q_3(z):=\Phi(-3;5;z)=1-\frac{3z}{5}+\frac{z^2}{10}-\frac{z^3}{210}\end{gather*}
are in the class $\mathscr{P}_e$. The subordinations $q_i(z)\prec e^z$ ($i=1,2,3$)
are graphically represented in Figure \ref{fig:conf}.

\begin{figure}[h]
\begin{center}
  \subfigure[$a=-1,c=3$]{\includegraphics[width=1.95in]{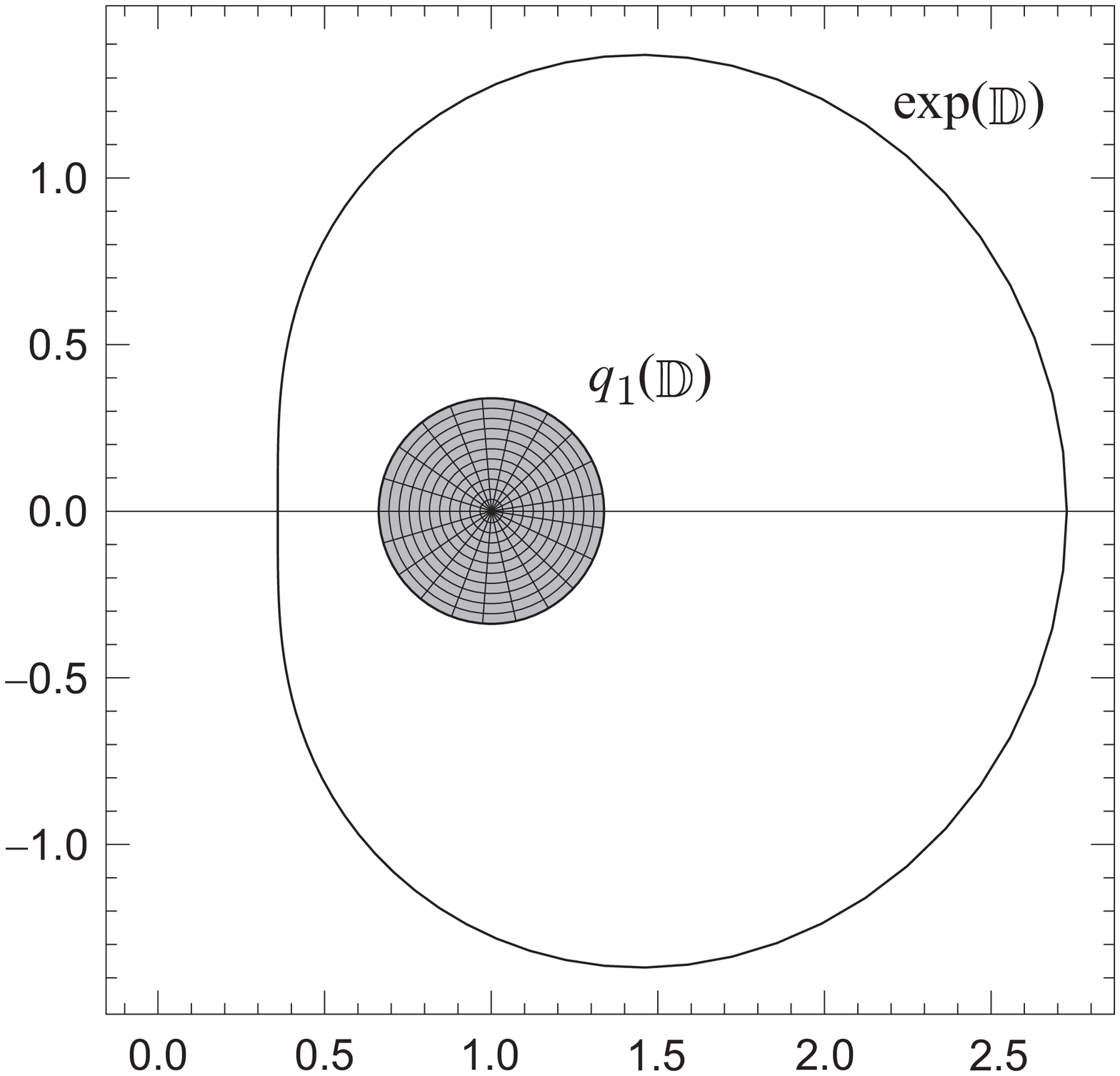}}\hspace{10pt}
  \subfigure[$a=-2,c=4$]{\includegraphics[width=1.95in]{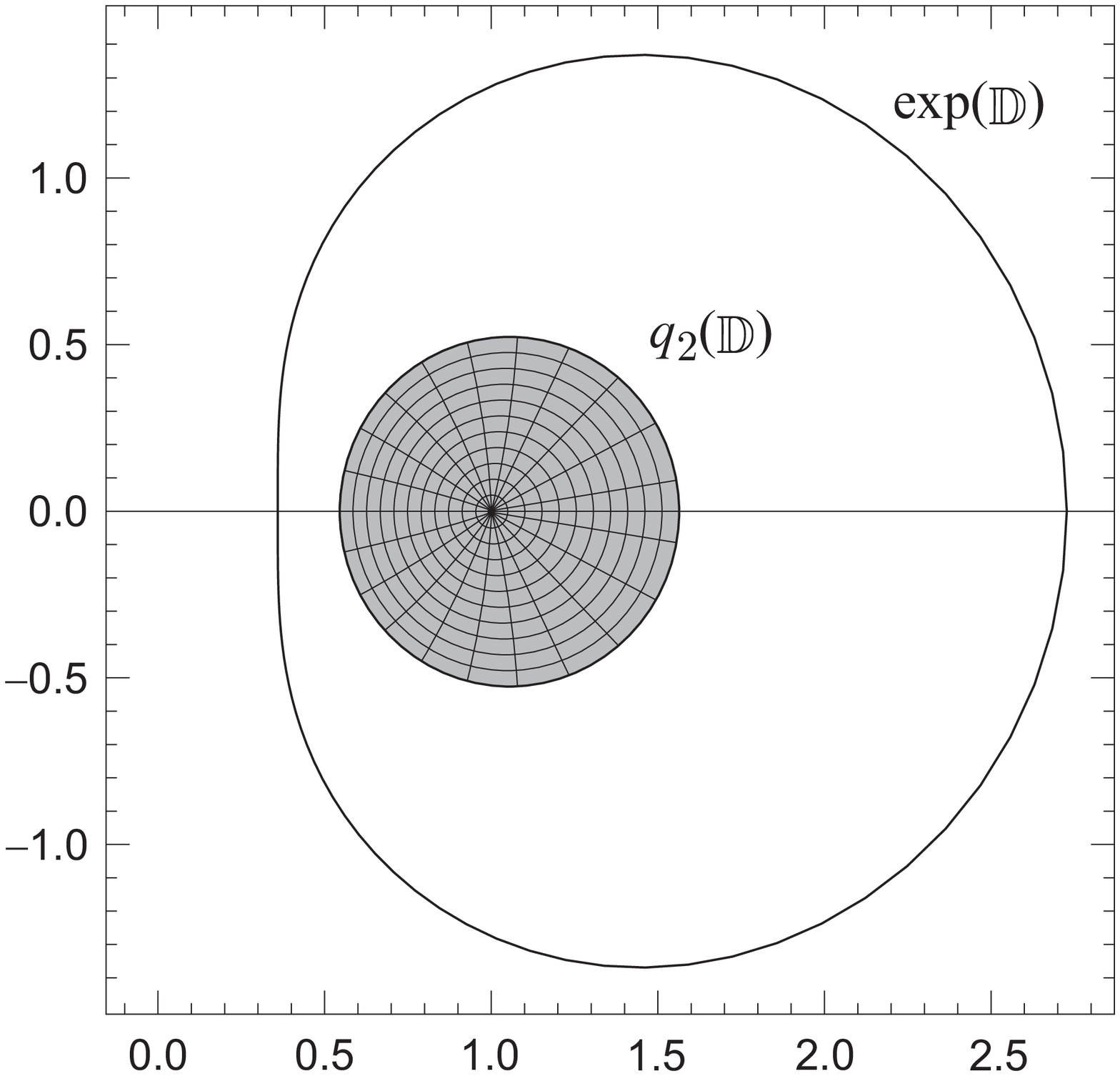}}\hspace{10pt}
  \subfigure[$a=-3,c=5$]{\includegraphics[width=1.95in]{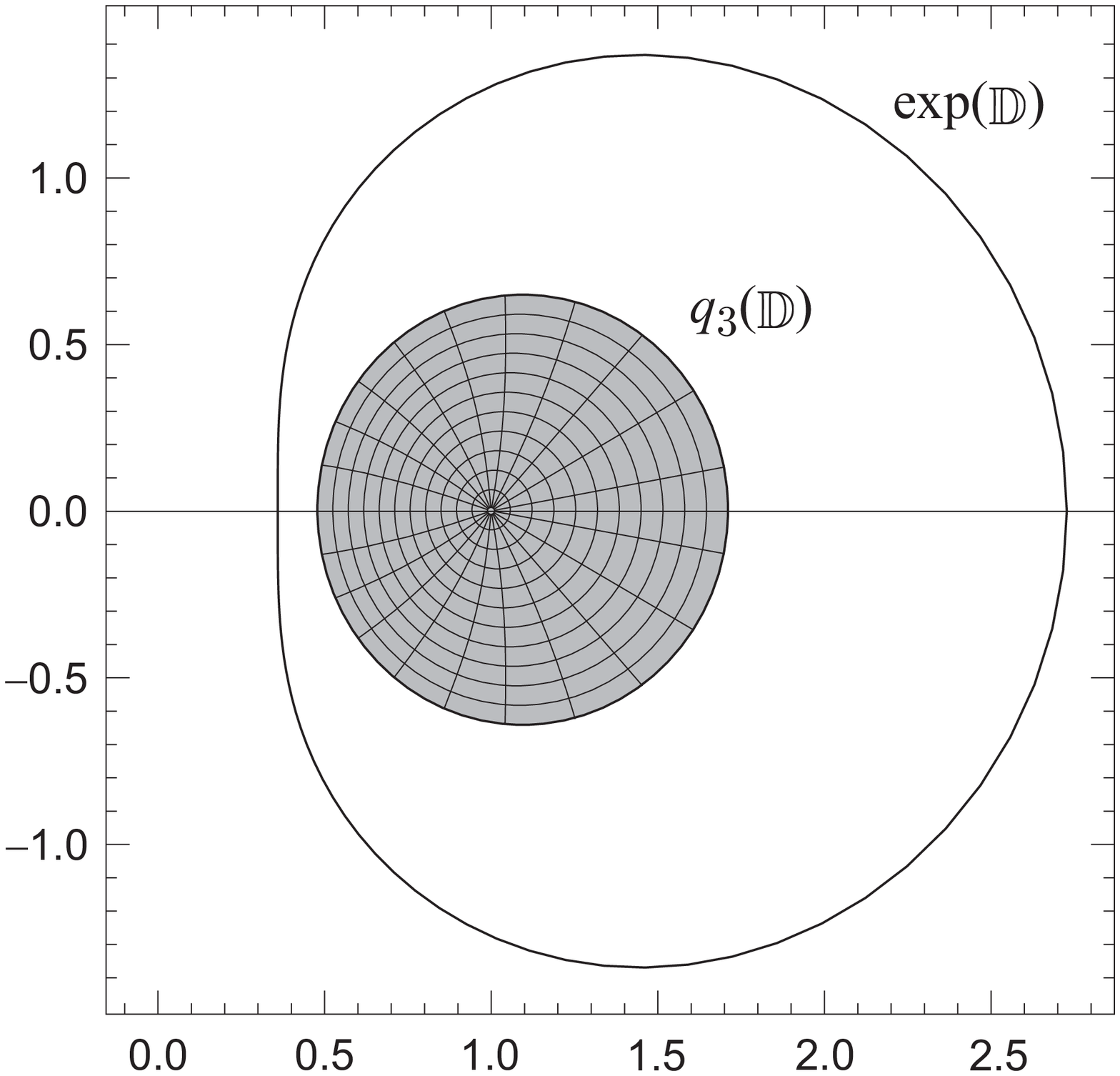}}
  \caption{Graph showing $q_i(\mathbb{D})\subset \exp(\mathbb{D})$.}\label{fig:conf}
\end{center}
\end{figure}
Similarly, the subordinations $\Phi(-25;27;z)\prec e^z$ and $\Phi(-100;102;z)\prec e^z$ are graphically shown in Figure \ref{fig:phi}.
\begin{figure}[h]
	\begin{center}
		\subfigure[$\Phi(-25;27;z)\prec e^z$]{\includegraphics[width=1.95in]{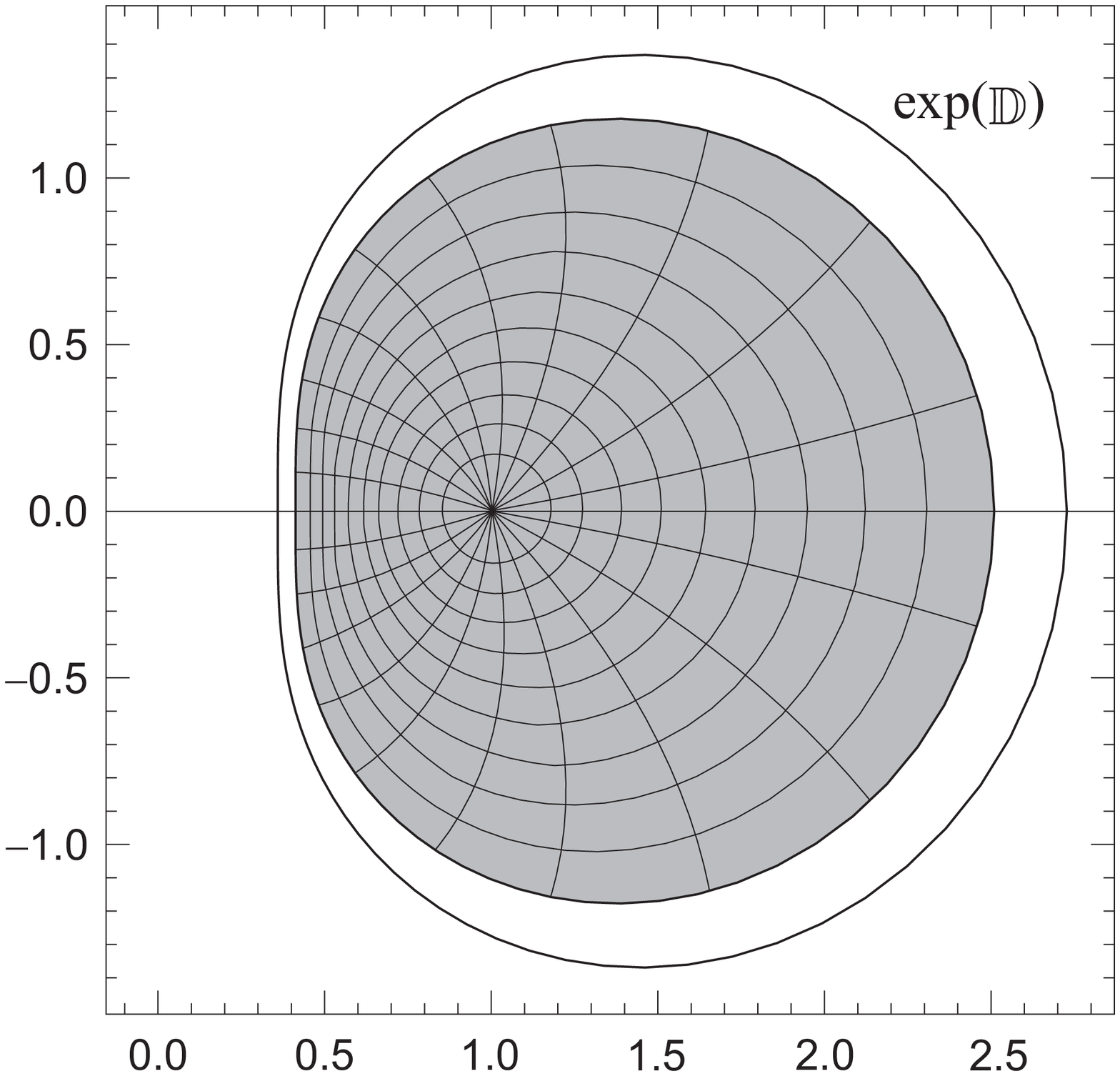}}\hspace{10pt}
		\subfigure[$\Phi(-100;102;z)\prec e^z$]{\includegraphics[width=1.95in]{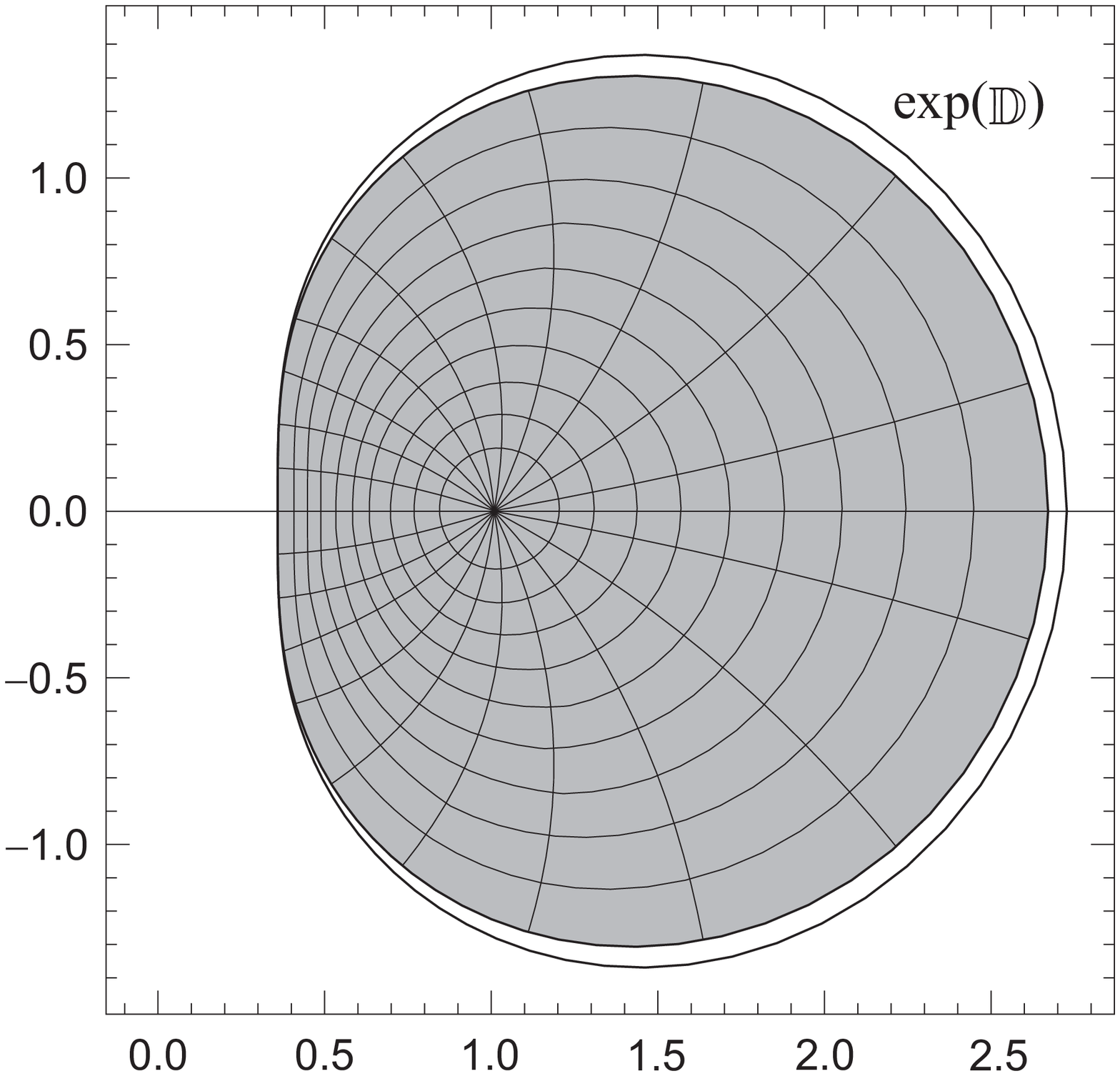}}\hspace{10pt}
		\caption{}\label{fig:phi}
	\end{center}
\end{figure}
\end{example}
The next theorem gives a sufficient condition on the parameters $a$ and $c$ for the confluent hypergeometric function $\Phi(a;c;z)$ to be exponential convex.

\begin{theorem}\label{thm}
Let the parameters $0\ne a$, $c\in\mathbb{R}$  be constrained such that $c$ is not a nonnegative integer and either (i) $a>-1$ and $c\geq a$ or (ii) $a\leq-1$ and $c\geq (1+(1+a)^2)^{1/2}$. If such $a$ and $c$ satisfy the following condition:
	\begin{equation}\label{con2}
(e-1)|c-2|+|a| \leq\frac{(e-1)^2(e+1)}{e^2}
	\end{equation} then the function $(\Phi(a;c;z)-1)c/a\in\ke$.
\end{theorem}

\begin{proof}
Let $\Lambda(a;c;z)=(\Phi(a;c;z)-1)c/a$ and define a function $p\colon\mathbb{D}\to\mathbb{C}$ by \[p(z)=1+\frac{z\Lambda''(a;c;z)}{\Lambda'(a;c;z)}=1+\frac{z\Phi''(a;c;z)}{\Phi'(a;c;z)}. \]
The conditions given in the hypothesis imply that the function $\re(c/a)\Phi'(a;c;z)>0$ for all $z\in\mathbb{D}$ by \cite[Theorem 1, p.~336]{MR1017006}. Therefore $\Phi'(a;c;z)\ne0$ for all $z\in\mathbb{D}$ and hence the function $p$ is analytic in $\mathbb{D}$. First, suppose that $z\ne0$. Since the function $\Phi(a;c;z)$ satisfies the differential equation \eqref{con1},  we have
\begin{equation*}
z\Phi''(a;c;z)+(c-z)\Phi'(a;c;z)-a\Phi(a;c;z)=0.
\end{equation*}	
Differentiation gives
\begin{equation*}
z\Phi'''(a;c;z)+(c-z+1)\Phi''(a;c;z)-(a+1)\Phi'(a;c;z)=0.
\end{equation*}
Since $\Phi'(a;c;z)\ne0$ for all $z\in\mathbb{D}$, dividing this equation by $\Phi'(a;c;z)$ and then multiplying it by $z$ yield
\begin{equation}\label{con3}
\frac{z^2\Phi'''(a;c;z)}{\Phi'(a;c;z)}+(c-z+1)\frac{z\Phi''(a;c;z)}{\Phi'(a;c;z)}-(a+1)z=0.
\end{equation}
Differentiating the equation $p(z)-1=z\Phi''(a;c;z)/\Phi'(a;c;z)$ logarithmically, we obtain
\begin{equation*}
\frac{z\Phi'''(a;c;z)}{\Phi''(a;c;z)}=\frac{zp'(z)+p^2(z)-3p(z)+2}{p(z)-1}
\end{equation*}
so that
\[\frac{z^2\Phi'''(a;c;z)}{\Phi'(a;c;z)}=zp'(z)+p^2(z)-3p(z)+2.\]
Putting this expression in \eqref{con3}, it follows that the function $p$ satisfies the following second order differential equation
\begin{equation}\label{con4}
zp'(z)+p^2(z)-1+(c-2)(p(z)-1)-(a+p(z))z=0.
\end{equation}
This equation is also valid for $z=0$. Now if we define a function $\Psi\colon\mathbb{C}^2\times\mathbb{D}\to\mathbb{C}$ by
\[\Psi(r,s;z):= s+r^2-1+(c-2)(r-1)-(a+r)z \] and let $\Omega:=\{0\}$, then the  second order differential equation \eqref{con4} can be rewritten as \[\Psi(p(z),zp'(z);z)\in\Omega\qquad (z\in\mathbb{D}). \]
Observe that
\begin{align}
|\Psi(r,s;z)|& =|s+r^2-1+(c-2)(r-1)-az-rz | \notag\\
& \geq |s+r^2-1|-|c-2|\cdot|r-1|-|a|\cdot|z |-|r|\cdot|z|.\label{con5}
\end{align}
For $r = e^{e^{i\theta}}$, $s = m e^{i\theta} e^{e^{i\theta}}$, $\theta\in[0,2\pi)$, $z\in \mathbb{D}$ and $m \geq 1$, we have
\begin{align*}|s+r^2-1|^2& = (\re (me^{i\theta}e^{e^{i\theta}}+e^{2e^{i\theta}}-1))^2+(\imag (me^{i\theta}e^{e^{i\theta}}+e^{2e^{i\theta}}-1))^2\\ &=(me^{\cos\theta}\cos(\theta+\sin\theta)+e^{2\cos\theta}\cos(2\sin\theta)-1)^2\\
&\quad\;+(me^{\cos\theta}\sin(\theta+\sin\theta)+e^{2\cos\theta}\sin(2\sin\theta))^2:=g(\theta).
\end{align*}
The function $g$ attains its minimum at $\theta=\pi$ by applying the second derivative test and
\[\min_{\theta\in[0,2\pi)}g(\theta)=g(\pi)=\left(\frac{-m}{e}+\frac{1}{e^2}-1\right)^2\]
so that
\[|s+r^2-1|\geq \frac{m}{e}-\frac{1}{e^2}+1 \geq  \frac{1}{e}-\frac{1}{e^2}+1.\]
Also,
\begin{align*}
 |r-1|^2&=(e^{\cos\theta}\cos(\sin\theta)-1)^2+e^{2\cos\theta}\sin^2(\sin\theta)\\
 &=e^{2\cos\theta}+1-2e^{\cos\theta}\cos(\sin\theta):=h(\theta).
 \end{align*}
The function $h$ attains its maximum at $\theta=0$ and hence $|r-1|\leq e-1$. Moreover, $|r|=e^{\cos\theta}\leq e$. Using \eqref{con5} and the above estimates, we have
\begin{align*}
|\Psi(r,s;z)|= |\Psi(e^{e^{i\theta}},me^{i\theta}e^{e^{i\theta}};z)| > \frac{1}{e}-\frac{1}{e^2}+1-(e-1)|c-2|-|a|-e\geq 0
\end{align*} whenever \eqref{con2} holds. Therefore $|\Psi(r,s;z)|>0$ and hence using Lemma \ref{lemA}, we conclude that $p(z)\prec e^z$ or $\Lambda\in \ke$.
\end{proof}

Using the famous Alexander duality theorem between the two classes $\ke$ and $\se$ which says that $f\in\ke$ if and only if $zf'\in\se$, the property \eqref{elm} of the function $\Phi(a;c;z)$: \[(a-1)z\Phi(a;c;z)=(c-1)z\Phi'(a-1;c-1;z)\]
and carrying out the case $a=1$ separately using the limit procedure (as done in Theorem \ref{conf1}), the sufficient condition for a function $z\Phi(a;c;z)$ to be exponential starlike is obtained.

\begin{corollary}\label{cor2}
	Let the parameters $a$, $c\in\mathbb{R}$  be constrained such that $c$ is not a nonnegative integer and either (i) $a> 0$ and $c\geq a$ or (ii) $a\leq 0$ and $c\geq 1+(1+a^2)^{1/2}$.
	If such $a$ and $c$ satisfy the following condition:
	\begin{equation*}
	(e-1)|c-3|+|a-1|\leq\frac{(e-1)^2(e+1)}{e^2}
	\end{equation*} then the function $z\Phi(a;c;z)\in\se$.
\end{corollary}

Let us illustrate Theorem \ref{thm} and Corollary \ref{cor2} by an example.

\begin{example}
The constants $a=1$ and $c=2$ satisfy the conditions of Theorem \ref{thm}. This gives  \[\Lambda(1;2;z)=2(\Phi(1;2;z)-1)=2\left(\frac{e^z-z-1}{z}\right)\in\ke.\]
Similarly, the constants $a=2$ and $c=3$ satisfy Corollary \ref{cor2} so that the function
\[\Upsilon(2;3;z)=z\Phi(2;3;z)=\frac{2+2(-1+z)e^z}{z}\in \se.\]
This can be seen graphically by considering the quantities
 \[1+\frac{z\Lambda''(1;2;z)}{\Lambda'(1;2;z)}\quad \mbox{and}\quad \frac{z\Upsilon'(2;3;z)}{\Upsilon(2;3;z)}. \]
Both these expressions turn out to  the function
\[q(z)=\frac{e^z(1-z+z^2)-1}{e^z(-1+z)+1}\]
which maps $\mathbb{D}$ inside $\exp(\mathbb{D})$ as depicted in Figure \ref{fig:conf2}.
\begin{figure}[h]
	\centering
	\includegraphics[scale=0.35]{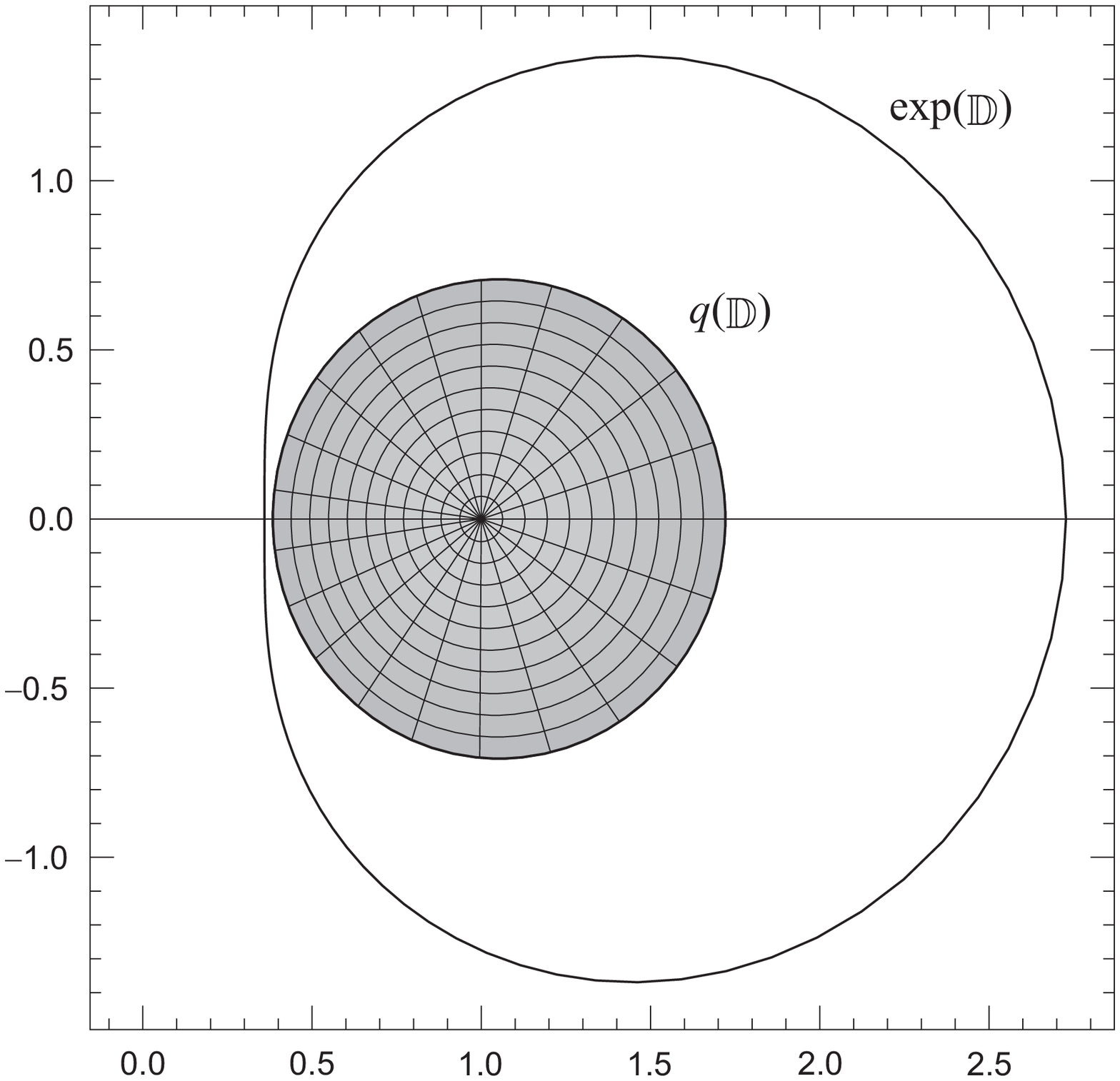}
	\caption{Graph showing $q(\mathbb{D})\subset \exp(\mathbb{D})$}
	\label{fig:conf2}
\end{figure}
\end{example}
For $\re(a)>0$ and $\re(c)>0$, the function $\Phi(a;c;z)$ also has following the integral representation \cite[Equation (1.2-8), p.~5]{MR1760285}:
\begin{equation*}
\Phi(a;c;z)=\frac{\Gamma(c)}{\Gamma(a)\cdot\Gamma(c-a)}\int_0^1t^{a-1}(1-t)^{c-a-1}e^{tz}dt=\int_0^1 e^{tz}d\mu(t)
\end{equation*} where\[d\mu(t)=\frac{\Gamma(c)t^{a-1}(1-t)^{c-a-1}}{\Gamma(a)\cdot\Gamma(c-a)}dt \]
is the probability measure on $[0,1]$.  Therefore using Theorem \ref{thm}, we have
\begin{equation*}
g_\delta(z):=\Phi(1;1+\delta;z)=\delta\int_0^1(1-t)^{\delta-1}e^{tz}dt
\end{equation*}belongs to the class $\ke$ when $|\delta-1|\leq(e^3  - 2 e^2 - e + 1)/(e^2(e-1))$. Similarly using Corollary \ref{cor2}, the function
\begin{equation*}
h_\delta(z):=z\Phi(1;1+\delta;z)=\delta z\int_0^1(1-t)^{\delta-1}e^{tz}dt
\end{equation*} belongs to the class $\se$ when $|\delta-2|\leq(e^2-1)/e^2$.

\section{Lommel Function of the First Kind}\label{lommel}
Prajapat \cite{MR2826152} obtained the sufficient conditions for the generalized and normalized Bessel function to be univalent in $\mathbb{D}$ while Baricz and Ponnusamy \cite{MR2743533} investigated the convexity and starlikeness for the same.  Kanas \textit{et al.\@}  \cite{MR123456} and Radhika \textit{et al.\@}  \cite{MR3753028} obtained the relation between the generalized Bessel function and the Janowski class using distinct techniques.
Very recently, Naz \textit{et al.\@}   \cite {MR0000} determined the conditions on the parameters of the generalized and normalized Bessel function to be exponential convex and exponential starlike. In this section, the connection of Lommel function of first kind with the classes $\ke$ and $\se$ is established.

Consider the second-order inhomogeneous Bessel differential equation
\begin{equation*}
z^2 \omega''(z)+z\omega'(z)+(z^2-\nu^2)\omega(z)=z^{\mu+1} \qquad (\mu,\nu,z\in\mathbb{C}).
\end{equation*}
Its particular solution known as the \emph{Lommel function of the first kind}, is denoted by $\s$ and  can be explicitly expressed in terms of hypergeometric function ${}_1F_2$ as
\begin{equation*}
\s(z)= \frac{z^{\mu+1}}{(\mu-\nu+1)(\mu+\nu+1)} {}_1F_{2} \left(1; \frac{\mu-\nu+3}{2}, \frac{\mu+\nu+3}{2};-\frac{z^2}{4} \right)
\end{equation*}
where $\mu\pm\nu$ is not  a negative odd integer. Since the function $\s$ does not belong to the class $\mathscr{A}$, following normalization of the Lommel function is taken into consideration
\begin{equation*}
\h(z)=(\mu-\nu+1) (\mu+\nu+1) z^{(1-\mu)/2} \s(\sqrt{z}) \qquad (z\in\mathbb{D})
\end{equation*}
which can be expressed in the infinite series
\begin{equation*}
\h (z)= z+\sum_{n\geq 1} \frac{(-1/4)^n}{((\mu-\nu+3)/2)_n ((\mu+\nu+3)/2)_n}z^{n+1} \qquad (z\in\mathbb{D}).
\end{equation*}
Therefore the function $\h\in\mathscr{A}$ and satisfies the  second-order differential equation
\begin{equation}\label{lom1}
z^2 \mathfrak{h}_{\mu,\nu}''(z)+\mu z\mathfrak{h}_{\mu,\nu}'(z)+\frac{1}{4} ( (\mu-1)^2 -\nu^2+z)\mathfrak{h}_{\mu,\nu}(z)= \frac{1}{4}((\mu+1)^2 -\nu^2)z.
\end{equation}
Ya\u{g}mur \cite{MR3353311} and Sim \textit{et al.\@}  \cite{sim2018geometric} investigated the geometric properties like convexity  and starlikeness of order $\alpha$ of normalized form of Lommel functions of the first kind. Sim \textit{et al.\@}  \cite{sim2018geometric} used the method of admissible functions to prove their results. Ya\u{g}mur \cite{MR3353311} also obtained a sufficient condition under which $\h$ becomes close-to-convex of order $(1+\alpha)/2$. Our first result gives the condition on the parameters $\mu$ and $\nu$ so that the Lommel function of the first kind $\h$ belongs to class of exponential convex functions.
\begin{theorem}\label{lom4}
Let the parameters $\mu$, $\nu\in\mathbb{R}$  be constrained such that $\mu\pm\nu$ are not  negative odd integers, $M=(\mu+5)^2-\nu^2$ and $N=(\mu+3)^2-\nu^2$. If such $\mu$ and $\nu$ satisfy the following three conditions:
\[\mu >-5+(3/2+\nu^2)^{1/2}, \quad \frac{4M}{N}<2M-3\]
and
	\begin{equation}\label{lom3}
	\mu(1+2\sin(1))-\frac{1}{4}e(e-1)\left|(\mu+1)(\mu-7)-\nu^2 \right| \geq e^4-3e^3+\frac{13 e^2}{4}-\frac{3e}{4}-\frac{3}{e}+2-2\sin (1)
	\end{equation} then $\h\in\ke$.
\end{theorem}
\begin{proof}
Let us define a function $p\colon\mathbb{D}\to\mathbb{C}$ by \[p(z)=1+\frac{z \mathfrak{h}_{\mu,\nu}''(z)}{\mathfrak{h}_{\mu,\nu}'(z)}. \]
The first condition implies that $M>3/2$. Also, by means of \cite[p.~1040]{MR3353311}, we have \[|\mathfrak{h}_{\mu,\nu}'(z)| \geq \frac{2 M N-4M-3N}{N(2M-3)}=1- \frac{4M}{N(2M-3)}>0\] by the second condition. Therefore $\mathfrak{h}_{\mu,\nu}'(z)\ne0$ for all $z\in\mathbb{D}$ and hence $p$ is an analytic function with $p(0)=1$. Assume that $z\ne0$.  Differentiating \eqref{lom1}, dividing it by $\mathfrak{h}_{\mu,\nu}'$ and then multiplying the obtained equation by $z$, we have
	\begin{equation*}
	 \frac{z^3\mathfrak{h}_{\mu,\nu}^{(4)}(z)}{\mathfrak{h}_{\mu,\nu}'(z)}+(\mu+4) \frac{z^2\mathfrak{h}_{\mu,\nu}'''(z)}{\mathfrak{h}_{\mu,\nu}'(z)}+\frac{1}{4} ( (\mu+3)^2 -\nu^2+z)\frac{z\mathfrak{h}_{\mu,\nu}''(z)}{\mathfrak{h}_{\mu,\nu}'(z)}+\frac{z}{2}=0.
	\end{equation*}
	 Differentiating the equation $z\mathfrak{h}''_{\mu,\nu}(z)/\mathfrak{h}'_{\mu,\nu}(z)=p(z)-1$ logarithmically and then multiplying $z$ on both sides, we see that the function $p$ satisfies the following second-order differential equation:
	 \begin{gather*}
	z^2p''(z)-5zp'(z)+3zp(z)p'(z)+p^3(z)-6p^2(z)+11 p(z)-6+(\mu+4)(zp'(z)\\ +\;p^2(z)-3p(z)+2)+\frac{1}{4}(p(z)-1)((\mu+3)^2-\nu^2 +z)+\frac{z}{2}=0.
	 \end{gather*}
After rearrangement of terms, we obtain
	  \begin{equation}\begin{split}\label{lom2}
	 z^2p''(z)+zp'(z)+(\mu-2)zp'(z)+3zp(z)p'(z)+(\mu+1)(p^2(z)-1)\\+\;(p(z)-1)^3+ \frac{1}{4}(p(z)-1)((\mu+1)(\mu-7)-\nu^2 )+\frac{1}{4}(p(z)+1)z=0.\end{split}
	 \end{equation}
	 Note that the above equation is also valid for $z=0$.
Now define a function $\Psi\colon\mathbb{C}^3\times\mathbb{D}\to\mathbb{C}$ by  \begin{align*}
	\Psi(r,s,t;z)&:= t+s+(\mu-2)s+3rs+(\mu+1)(r^2-1)+(r-1)^3\\
	&\quad +\frac{1}{4}(r-1)((\mu+1)(\mu-7)-\nu^2 )+\frac{1}{4}(r+1)z\end{align*}
and suppose that $\Omega:=\{0\}$. Then 	\eqref{lom2} can be written as \[\Psi(p(z),zp'(z),z^2p''(z);z)\in\Omega\qquad (z\in\mathbb{D}). \]
 We will use  Lemma \ref{lemA} to prove the required result, that is, we  show that  $\Psi(r,s,t;z)\notin\Omega$ whenever $r = e^{e^{i\theta}}$,   $s = m e^{i\theta} e^{e^{i\theta}}$ and $\re ((s+t)e^{-i\theta}e^{-e^{i\theta}}) \geq 0 $ where $z\in\mathbb{D}$,  $\theta \in [0,2\pi)$ and $m \geq 1$ using the given condition \eqref{lom3}.  Consider
	\begin{align*}
	 |\Psi(r,s,t;z)| & = \left| e^{e^{i\theta}}\left[ (t+s)e^{-e^{i\theta}} +(\mu-2) m e^{i\theta}+3m e^{i\theta}e^{e^{i\theta}} +(\mu+1)\left(e^{e^{i\theta}}-e^{-e^{i\theta}}\right)\right]\right.\\
	 &\quad  +\left. (e^{e^{i\theta}}-1)^3+\frac{1}{4} (e^{e^{i\theta}}-1)((\mu+1)(\mu-7)-\nu^2)+\frac{1}{4} (e^{e^{i \theta}}+1)z\right|\\
	& \geq e^{\cos\theta}\left|(t+s) e^{-i\theta}e^{-e^{i\theta}}+(\mu-2)m+3 me^{e^{i\theta}}+(\mu+1)(e^{e^{i\theta}}-e^{-e^{i\theta}})e^{-i\theta}\right|\\
	& \quad -|e^{e^{i\theta}}-1|^3-\frac{1}{4}|e^{e^{i\theta}}-1|\cdot |(\mu+1)(\mu-7)-\nu^2|-\frac{1}{4}|e^{e^{i\theta}}+1|\cdot|z|\\
&\geq \frac{1}{e} \left[\re ((t+s) e^{-i\theta}e^{-e^{i\theta}})+ (\mu-2)m + 3 m \re (e^{e^{i\theta}}) +(\mu+1)\re\big((e^{e^{i\theta}}-e^{-e^{i\theta}})e^{-i\theta}\big)  \right]\\
	& \quad -|e^{e^{i\theta}}-1|^3-\frac{1}{4}|e^{e^{i\theta}}-1|\cdot |(\mu+1)(\mu-7)-\nu^2|-\frac{1}{4}|e^{e^{i\theta}}+1|\cdot|z|\\
&\geq \frac{1}{e} \left[(\mu-2)m + 3 m \re (e^{e^{i\theta}}) +(\mu+1)\re\big((e^{e^{i\theta}}-e^{-e^{i\theta}})e^{-i\theta}\big)  \right]\\
	& \quad -|e^{e^{i\theta}}-1|^3-\frac{1}{4}|e^{e^{i\theta}}-1|\cdot |(\mu+1)(\mu-7)-\nu^2|-\frac{1}{4}|e^{e^{i\theta}}+1|\cdot|z|
	 \end{align*}
Since $|e^{e^{i\theta}}-1| \leq e-1$, $|e^{e^{i\theta}}+1|\leq e+1$, $\re (e^{e^{i\theta}})=e^{\cos\theta}\cos(\sin\theta)\geq 1/e$ and $\mu>2$ (using \eqref{lom3}), it follows that
\begin{align*}
 |\Psi(r,s,t;z)| & > \frac{1}{e} \left[\mu-2 + \frac{3}{e} +(\mu+1)g(\theta)  \right]-(e-1)^3\\
 &\quad-\frac{1}{4}(e-1)|(\mu+1)(\mu-7)-\nu^2|-\frac{1}{4}(e+1)
\end{align*}
where the function $g$ is defined as
\begin{align*} g(\theta)& := \re \big((e^{e^{i\theta}}-e^{-e^{i\theta}})e^{-i\theta}\big)\\
&= e^{\cos\theta}\cos(\theta-\sin\theta)-e^{-\cos\theta}\cos(\theta+\sin\theta).
\end{align*}
Critical points of $g$ are $0$, $\pi/2$, $\pi$ and $3\pi/2$. Since $g''(0)=g''(\pi)=3/e-e<0$ and $g''(\pi/2)=g''(3\pi/2)=2 \cos(1)>0$, the second derivative test verifies that the minimum value of $g$ occurs at either $\pi/2$ or $3\pi/2$. Hence\[\min_{\theta\in[0,2\pi)} g(\theta) = g\left(\frac{\pi}{2}\right)=g\left(\frac{3\pi}{2}\right)=2\sin(1).\]
Thus we conclude that
\begin{align*}
|\Psi(r,s,t;z)| & > \frac{1}{e} \left[ \mu-2+\frac{3}{e}+2(\mu+1)\sin(1)\right] -(e-1)^3\\&\quad -\;\frac{1}{4}(e-1)|(\mu+1)(\mu-7)-\nu^2|-\frac{1}{4}(e+1)\geq 0
\end{align*} using the given hypothesis \eqref{lom3}.  This proves $|\Psi(r,s,t;z)|\ne0$. Therefore by means of Lemma \ref{lemA}, we have $p(z)\prec e^z$ for all $z\in\mathbb{D}$ which implies $\h\in\ke$.
\end{proof}
Consider the Alexander transform $\mathfrak{f}_{\mu,\nu}\colon\mathbb{D}\to\mathbb{C}$ of the function $\h$ defined by \begin{equation*} \mathfrak{f}_{\mu,\nu}(z):=\int_{0}^{z} \frac{\h(t)}{t}\, dt.\end{equation*} Note that $\mathfrak{f}_{\mu,\nu}\in\mathscr{A}$. The next theorem gives a sufficient condition under which the function $\mathfrak{f}_{\mu,\nu}$ becomes exponential convex and $\h$ becomes exponential starlike.
\begin{theorem}
	Let the parameters $\mu$, $\nu\in\mathbb{R}$  be constrained such that $\mu\pm\nu$ are not  negative odd integers and $	(\mu+1) ((\mu+1)(\mu+3)-\nu^2)\geq1/8$.
	 If such $\mu$ and $\nu$ satisfy the following condition:
	\begin{equation}\label{lom6}
\mu(2e-1)-\frac{1}{4}e(e-1)\left|(\mu-1)^2-\nu^2 \right| \geq e^3-e^2+\frac{13 e}{4} -4
	\end{equation} then the function $\f\in\ke$ and hence $\h\in\se$.
\end{theorem}
\begin{proof}Define a function $p\colon \mathbb{D}\to\mathbb{C}$ by
\begin{equation}\label{ee}p(z)=1+\frac{z \mathfrak{f}_{\mu,\nu}''(z)}{\mathfrak{f}_{\mu,\nu}'(z)} = \frac{z\mathfrak{h}'_{\mu,\nu}(z)}{\h(z)}. \end{equation}
Since the parameters $\mu$ and $\nu$ satisfy \eqref{lom6}, we must have $\mu>-1$ (in fact, $\mu>3$) and the condition  $	(\mu+1) ((\mu+1)(\mu+3)-\nu^2)\geq1/8$ implies that $\re (\mathfrak{f}'_{\mu,\nu}(z))=\re(\h(z)/z)>0$ using \cite[Corollary 2.4, p.~1042]{MR3353311}. Therefore the function $p$ is analytic in $\mathbb{D}$ and satisfies $p(0)=1$. As the function $\h$ satisfies \eqref{lom1}, hence by making use of \eqref{ee}, we obtain the equation
\begin{equation*}
\left[zp'(z)+p^2(z)+(\mu-1)p(z)+\frac{1}{4}((\mu-1)^2-\nu^2+z) \right] \h (z)=\frac{1}{4} ((\mu+1)^2-\nu^2)z.
\end{equation*}
Set $q(z):= zp'(z)+p^2(z)+(\mu-1)p(z)+((\mu-1)^2-\nu^2+z)/4 $. Thus the above equation can be rewritten as $q(z)\h(z)= ((\mu+1)^2-\nu^2)z/4$.  Differentiating this equation and then multiplying it by $z$ yields $(zq'(z)+(p(z)-1)q(z))\h(z)=0$. Since $\h(z)\ne0$ for all $z\in\mathbb{D}\setminus\{0\}$ while $zq'(z)+(p(z)-1)q(z)=0$ for $z=0$, we must have $zq'(z)+(p(z)-1)q(z)=0$ for all $z\in\mathbb{D}$. Therefore the function $p$ satisfies the second-order differential equation
\begin{gather*}
z^2p''(z)+2z p(z)p'(z)+(\mu-1)zp'(z)+zp'(z)+(p(z)-1)((\mu-1)p(z)\\+\;p^2(z)+zp'(z))+\frac{1}{4} (p(z)-1) ((\mu-1)^2-\nu^2+z)+\frac{z}{4}=0.
\end{gather*}
Rearrangement of terms leads to the following equation
\begin{equation}\begin{split}\label{lom5}
z^2p''(z)+zp'(z)+(\mu-2)zp'(z)+3 p(z)zp'(z)+(\mu-1)p(z)(p(z)-1)\\+\;p^2(z)(p(z)-1)+\frac{1}{4} (p(z)-1) ((\mu-1)^2-\nu^2)+\frac{1}{4}zp(z)=0.
\end{split}\end{equation}
Let $\Omega:=\{0\}$ and define a function $\Psi\colon\mathbb{C}^3\times\mathbb{D}\to\mathbb{C}$ by \begin{align*}
\Psi(r,s,t;z)& :=  t+s+(\mu-2)s+3rs+(\mu-1)r(r-1)+r^2(r-1)\\
&\quad  + \frac{1}{4}(r-1)((\mu-1)^2-\nu^2)+\frac{1}{4}rz.
\end{align*}
By \eqref{lom5}, we have $\Psi(p(z),zp'(z),z^2p''(z);z)\in\Omega$ for all $z\in\mathbb{D}$. Again, we will apply Lemma~ \ref{lemA} to prove that $p(z)\prec e^z$. For $r = e^{e^{i\theta}}$,   $s = m e^{i\theta} e^{e^{i\theta}}$ and $\re ((s+t)e^{-i\theta}e^{-e^{i\theta}}) \geq 0 $ where $z\in\mathbb{D}$,  $\theta \in [0,2\pi)$ and $m\geq 1$, consider
\begin{align*}
|\Psi(r,s,t;z)|& = |e^{e^{i\theta}}|\cdot\bigg|(t+s)e^{-e^{i\theta}} +(\mu-2) m e^{i\theta} +3 m e^{i\theta}e^{e^{i\theta}} +(\mu-1)(e^{e^{i\theta}}-1)\\
& \quad+\, e^{e^{i\theta}}(e^{e^{i\theta}}-1)+\frac{1}{4}(1-e^{-e^{i\theta}})((\mu-1)^2-\nu^2)+\frac{z}{4}\bigg|\\
& \geq e^{\cos\theta} \bigg[\left| (t+s)e^{-i\theta}e^{-e^{i\theta}} +(\mu-2) m  +3 m e^{e^{i\theta}} +(\mu-1)(e^{e^{i\theta}}-1)e^{-i\theta}\right|\\
& \quad-\, \left.|e^{e^{i\theta}}|\cdot|e^{e^{i\theta}}-1|-\frac{1}{4}|1-e^{-e^{i\theta}}|\cdot|(\mu-1)^2-\nu^2|-\frac{|z|}{4}\right].
\end{align*}
Using the similar analysis as carried out in Theorem \ref{lom4}, it is easy to deduce that
\begin{align*}
|\Psi(r,s,t;z)|& > \frac{1}{e} \bigg[\re((t+s)e^{-i\theta}e^{-e^{i\theta}}) +(\mu-2) m  +3 m \re (e^{e^{i\theta}})+(\mu-1)h(\theta) \\& \quad-\;e (e-1)-\frac{1}{4} (e-1) |(\mu-1)^2-\nu^2|-\frac{1}{4}\bigg]\\
& \geq \frac{1}{e} \bigg[\mu-2  +\frac{3}{e} +(\mu-1)h(\theta) - e (e-1)-\frac{1}{4} (e-1) |(\mu-1)^2-\nu^2|-\frac{1}{4}\bigg]
\end{align*}
where $h(\theta)= \re((e^{e^{i\theta}}-1)e^{-i\theta})=e^{\cos\theta}\cos(\theta-\sin\theta)-\cos\theta$.  Now the values of $h''$ at the critical points are $h''(0)=1-e<0$ and $h''(\pi)=3/e-1>0$. Hence by the second derivative test, we obtain \[\min_{\theta\in[0,2\pi)}h(\theta)=h(\pi)=1-1/e \] and thus
\begin{align*}
|\Psi(r,s,t;z)|& > \frac{1}{e} \bigg[\mu-2  +\frac{{3}}{e}+(\mu-1)\left(1-\frac{1}{e}\right)- e (e-1)-\frac{1}{4} (e-1) |(\mu-1)^2-\nu^2|-\frac{1}{4}\bigg]\\
& = \frac{1}{e} \bigg[\mu\left(2-\frac{1}{e}\right)  +\frac{4}{e}- e (e-1)-\frac{1}{4} (e-1) |(\mu-1)^2-\nu^2|-\frac{13}{4}\bigg]\geq0
\end{align*} using \eqref{lom6}. This gives $\Psi(r,s,t;z)\notin\Omega$ and therefore Lemma \ref{lemA} gives the desired result.
\end{proof}
Now we find the condition on the parameters of $\f$ or $\h$ under which the functions $\mathfrak{f}'_{\mu,\nu}$ or $\h(z)/z$ belongs to the class $\mathscr{P}_e$.
\begin{theorem}\label{thm1}
	Let the parameters $\mu$, $\nu\in\mathbb{C}$  be constrained such that $\mu\pm\nu$ are not  negative odd integers.	If such $\mu$ and $\nu$ satisfy the following condition:
	\begin{equation}\label{lom9}
	4\re (\mu)\geq(e-1)\left|(\mu+1)^2-\nu^2 \right|-3
	\end{equation} then $\mathfrak{f}'_{\mu,\nu}\in\mathscr{P}_e$ or $\h(z)/z\in\mathscr{P}_e$.
\end{theorem}
\begin{proof}
Observe that the condition \eqref{lom9} implies that $\re(\mu)\geq-3/4$.	Let us define a function $p\colon\mathbb{D}\to\mathbb{C}$ by \[p(z)=\mathfrak{f}'_{\mu,\nu}(z)=\frac{\h(z)}{z}. \] Clearly the function $p$ is analytic in $\mathbb{D}$ with $p(0)=1$. Since the function $\h$ satisfies \eqref{lom1}, the function $p$ satisfies the second-order differential equation
	\begin{equation*}
	z^2p''(z)+(\mu+2)zp'(z)+\frac{1}{4}p(z)((\mu+1)^2-\nu^2+z) -\frac{1}{4}((\mu+1)^2-\nu^2)=0
	\end{equation*}
	which can be rewritten as
		\begin{equation*}
	z^2p''(z)+zp'(z)+(\mu+1)zp'(z)+\frac{1}{4}(p(z)-1)((\mu+1)^2-\nu^2)+\frac{1}{4}zp(z)=0.
	\end{equation*}
Define a function $\Psi\colon\mathbb{C}^3\times\mathbb{D}\to\mathbb{C}$ by
\begin{equation*}
\Psi(r,s,t;z):=t+s+(\mu+1)s+\frac{1}{4}(r-1)((\mu+1)^2-\nu^2) +\frac{1}{4}rz
\end{equation*}and suppose that  $\Omega:=\{0\}$. Then $\Psi(p(z),zp'(z),z^2p''(z);z)\in\Omega$ for all $z\in\mathbb{D}$. To apply Lemma \ref{lemA}, note that
\begin{align*}
|\Psi(r,s,t;z)|&=|e^{e^{i\theta}}|\cdot\left|(t+s)e^{-e^{i\theta}}+(\mu+1)m e^{i\theta}+\frac{1}{4}(1-{e^{-e^{i\theta}}}) ((\mu+1)^2-\nu^2)+\frac{z}{4} \right|\\
& \geq e^{\cos\theta}\left[\left| (t+s)e^{-i\theta}e^{-e^{i\theta}}+(\mu+1)m \right|-\frac{1}{4}|1-e^{-e^{i\theta}}|\cdot |(\mu+1)^2-\nu^2|-\frac{|z|}{4} \right] \\
& > \frac{1}{e}\left[\re ((t+s)e^{-i\theta}e^{-e^{i\theta}})+(\re(\mu)+1)m -\frac{1}{4}(e-1)|(\mu+1)^2-\nu^2|-\frac{1}{4} \right]\\
&\geq\frac{1}{e}\left[\re(\mu)+1-\frac{1}{4}(e-1)|(\mu+1)^2-\nu^2|-\frac{1}{4} \right]\geq 0
\end{align*} whenever $r = e^{e^{i\theta}}$,   $s = m e^{i\theta} e^{e^{i\theta}}$ and $\re ((s+t)e^{-i\theta}e^{-e^{i\theta}}) \geq 0 $ where $z\in\mathbb{D}$,  $\theta \in [0,2\pi)$ and $m \geq 1$. Therefore $\Psi(r,s,t;z)\not\in\Omega$ and Lemma \ref{lemA} completes the proof.
\end{proof}
As an example, since $\mu=1$ and $\nu=0$ satisfy the hypothesis of Theorem \ref{thm1}, the function \begin{equation*}\label{lom}\mathfrak{f}'_{1,0}(z)= \frac{4-4J_0(\sqrt{z})}{z}\prec e^z\end{equation*} where $J_\nu$ denote the Bessel function of first kind of order $\nu$ having the form \[J_\nu(z)=\sum_{n\geq 0}\frac{(-1)^n}{n! \Gamma(\nu+n+1)}\left(\frac{z}{2}\right)^{2n+\nu}.\]

\section{Generalized Struve Function of the First Kind}\label{struve}
Consider the second order inhomogeneous Bessel differential equation
\begin{equation}\label{st1}
z^2 \omega''(z)+z\omega'(z)+(z^2-\nu^2)\omega(z)=\frac{4(z/2)^{\nu+1}}{\sqrt{\pi}\Gamma(\nu+1/2)} \qquad (\nu,z\in\mathbb{C}).
\end{equation}
The particular solution of \eqref{st1} known as   the \emph{Struve function of the first kind  of order $\nu$}, is denoted by $\textbf{H}_\nu$ and can be expressed explicitly  in terms of the hypergeometric function as follows:
\begin{equation*}
\textbf{H}_\nu(z)= \frac{(z/2)^{\nu+1}}{\sqrt{(\pi/4)\Gamma(\nu+3/2)}} {}_1F_{2} \left(1; \frac{3}{2},\nu+\frac{3}{2};-\frac{z^2}{4} \right)
\end{equation*}
where $\nu+3/2$ is not a negative integer. The function $\textbf{H}_\nu$ also has the infinite series representation
\begin{equation}\label{st9}
 \textbf{H}_{\nu}(z)=\sum_{n\geq 0} \frac{(-1)^n}{ \Gamma (n+3/2)\Gamma(\nu+n+3/2)}\left(\frac{z}{2}\right)^{2n+\nu+1}\qquad (z\in\mathbb{C}).
\end{equation}
The second order differential equation
\begin{equation}\label{st3}
z^2\omega''(z)+z\omega'(z)-(z^2+\nu^2)\omega(z)=\frac{4(z/2)^{\nu+1}}{\sqrt{\pi}\Gamma(\nu+1/2)} \qquad (\nu,z\in\mathbb{C})
\end{equation}
 differs from \eqref{st1} only in coefficient of $\omega$. The particular solution of \eqref{st3} is called as   the \emph{modified Struve function of the first kind of order $\nu$} and is  defined  by
 \begin{equation}\label{st10}
 \textbf{L}_\nu(z)=-ie^{-i\nu\pi/2} \textbf{H}_\nu(iz)=\sum_{n\geq 0} \frac{1}{ \Gamma (n+3/2)\Gamma(\nu+n+3/2)}\left(\frac{z}{2}\right)^{2n+\nu+1}\,  (z\in\mathbb{C}).
 \end{equation}
 Now consider the second order inhomogeneous linear differential equation
 \begin{equation}\label{st2}
 z^2 \omega''(z)+bz\omega'(z)+(cz^2-\nu^2+(1-b)\nu)\omega(z)=\frac{4(z/2)^{\nu+1}}{\sqrt{\pi}\Gamma(\nu+b/2)} \qquad (b,c,\nu,z\in\mathbb{C}).
 \end{equation}
  The case $b=1$ and $c=1$ in \eqref{st2} leads to \eqref{st1} while the case $b=1$ and $c=-1$ gives \eqref{st3}. Therefore we can say that \eqref{st2} generalizes \eqref{st1} and \eqref{st3}. This permits  the study of the geometric properties of Struve and modified Struve functions in a unified manner. The particular solution of \eqref{st2} is called as the \emph{generalized Struve function of the first kind of order $\nu$} and is denoted by $ \mathfrak{w}_{\nu,b,c}$. The function $\mathfrak{w}_{\nu,b,c}$ has the infinite series representation
  \begin{equation}\label{st4}
  \mathfrak{w}_{\nu,b,c}(z)=\sum_{n\geq 0} \frac{(-c)^n}{ \Gamma (n+3/2)\Gamma(\nu+n+(b+2)/2)}\left(\frac{z}{2}\right)^{2n+\nu+1}\qquad (z\in\mathbb{C}).
  \end{equation} Despite the fact that the series \eqref{st4} is convergent in the whole complex plane, the function $\mathfrak{w}_{\nu,b,c}$ is not univalent in $\mathbb{D}$.   If we take into consideration the normalization of the function $\mathfrak{w}_{\nu,b,c}$ defined by the transformation \begin{equation}\label{phi}\mathfrak{u}_{\nu,b,c}(z)= 2^\nu \sqrt{\pi}\Gamma\left(\nu+\frac{b+2}{2}\right)z^{-(\nu+1)/2} \mathfrak{w}_{\nu,b,c}(\sqrt{z})\end{equation}then  the function $\mathfrak{u}_{\nu,b,c}$ has the infinite series representation \begin{equation}\label{e2}
\mathfrak{u}_{\nu,b,c}(z)=  \sum_{n\geq 0} \frac{(-c/4)^n}{(3/2)_n(\kappa)_n}z^n\qquad (z\in\mathbb{C})
\end{equation}
where $\kappa=\nu+(b+2)/2\ne 0,-1,-2,\ldots$.
For brevity, we shall denote $\mathfrak{u}_{\nu,b,c}$  simply by $\mathfrak{u}_{\nu}$.  Note that the function  $\mathfrak{u}_{\nu}$ is entire and satisfies the following second order differential equation: \begin{equation}\label{st5}
4z^2 \mathfrak{u}''_{\nu}(z)+2(2\kappa+1) z \mathfrak{u}'_{\nu}(z)+(cz+2(\kappa-1))\mathfrak{u}_{\nu}(z)-2(\kappa-1)=0.\end{equation}
Ya\u{g}mur and Orhan \cite{MR3035216} gave the sufficient conditions on the parameters of the generalized Struve function to be convex and starlike in $\mathbb{D}$. Using the technique of differential subordination, Orhan and Ya\u{g}mur \cite{MR3718596} obtained the conditions under which $\mathfrak{u}_{\nu}$ is univalent, convex, starlike and close-to-convex. Recently, 	Noreen \textit{et al.\@}  \cite{MR3927311} found the relationship between $\mathfrak{u}_{\nu}$ and the Janowski class.

The condition on the parameters $\kappa$ and $c$ is determined in the first theorem of this section for the generalized  Struve function $\mathfrak{u}_{\nu}$ to belong to the class $\mathscr{P}_e$.
\begin{theorem}\label{thmB}
	If the parameters $\kappa$, $c \in\mathbb{C}$ are constrained such that $\kappa$ is not a nonnegative integer and \begin{equation}\label{st7}
	\re (\kappa)-\frac{1}{2}(e-1) |\kappa-1|\geq \frac{|c|}{4}+\frac{1}{2}
	\end{equation} then $ \mathfrak{u}_{\nu} \in\mathscr{P}_e$.
\end{theorem}
\begin{proof}
Set $p(z):=\mathfrak{u}_{\nu}(z)$.  Clearly $p$ is an analytic function in $\mathbb{D}$ and $p(0)=1$. Since the function $\mathfrak{u}_{\nu}$ satisfies the second order differential equation \eqref{st5},  the  function $p$ also satisfies the differential equation
 \begin{equation*}
4z^2 p''(z)+2(2\kappa+1) z p'(z)+(cz+2(\kappa-1))p(z)-2(\kappa-1)=0.\end{equation*} Rearrangement of the terms yields
 \begin{equation}\label{st6}
4z^2 p''(z)+4zp'(z)+2(2\kappa-1) z p'(z)+2(\kappa-1)(p(z)-1)+czp(z)=0.\end{equation}
Define a function $\Psi(r,s,t;z):=4t+4s+2(2\kappa-1)s+2(\kappa-1)(r-1)+crz$ and let $\Omega:=\{0\}$. Then \eqref{st6} can be rewritten as \[\Psi(p(z),zp'(z),z^2p''(z);z)\in\Omega\qquad (z\in\mathbb{D}). \]
In order to obtain the desired result, we apply Lemma \ref{lemA}. For this, we show that  $\Psi(r,s,t;z)\notin\Omega$ whenever $r = e^{e^{i\theta}}$,   $s = m e^{i\theta} e^{e^{i\theta}}$ and $\re ((s+t)e^{-i\theta}e^{-e^{i\theta}}) \geq 0 $ where $z\in\mathbb{D}$,  $\theta \in [0,2\pi)$ and $m \geq 1$. Then
\begin{align*}
|\Psi(r,s,t;z)| & = |e^{e^{i\theta}}|\cdot \left|4(t+s)e^{-e^{i\theta}}+2(2\kappa-1) me^{i\theta}+2(\kappa-1)(1-e^{-e^{i\theta}})+ cz\right|\\
&> 2e^{\cos\theta}\left[|2(t+s)e^{-i\theta}e^{-e^{i\theta}}+(2\kappa-1) m|-|\kappa-1|\cdot|1-e^{-e^{i\theta}}|-\frac{|c|}{2}\right]\\
& \geq  \frac{2}{e}\left[2\re((t+s)e^{-i\theta}e^{-e^{i\theta}})+(2\re(\kappa)-1) m-(e-1)|\kappa-1| -\frac{|c|}{2}\right]\\
& \geq \frac{2}{e}\left[2\re(\kappa)-1-(e-1)|\kappa-1|-\frac{|c|}{2} \right]\geq0
\end{align*}
by using \eqref{st7} and its implication $\re(\kappa)\geq1/2$. Therefore by Lemma \ref{lemA}, we conclude that $p(z)\prec e^z$ for all $z\in\mathbb{D}$ which gives  $\mathfrak{u}_{\nu}\in\mathscr{P}_e$.
\end{proof}

As an illustration, if we choose $c=1$, $\nu=1$ and $b=0$, then $\kappa=2$, so that such $c$ and $\kappa$ satisfy the condition \eqref{st7}. Hence using Theorem \ref{thmB}, we have \[ \mathfrak{u}_{1}(z)=\frac{2-2\cos\sqrt{z}}{z}\prec e^z.\]

The generalized Struve function satisfies the following recursive relation \cite[Proposition 2.1(v), p.~6]{MR3718596}:
\begin{equation}
\mathfrak{u}_{\nu}(z)+2z\mathfrak{u}'_{\nu}(z)+\frac{cz}{2\kappa}\mathfrak{u}_{\nu+1}(z)=1.
\end{equation} Using this relation, we obtain the following corollary.
\begin{corollary}
		If the parameters $0\ne c$, $\kappa  \in\mathbb{C}$ are constrained such that  $\kappa$ is not a nonnegative integer and  \begin{equation*}
		\re (\kappa+1)-\frac{1}{2}(e-1) |\kappa| \geq \frac{|c|}{4}+\frac{1}{2}
		\end{equation*} then the function $ (2\kappa/cz)(1-2z\mathfrak{u}'_\nu(z)-\mathfrak{u}_{\nu}(z))\in\mathscr{P}_e$.
\end{corollary}

The next result deals with the sufficient condition on $\kappa$ and $c$ so that the generalized  Struve function belongs to the class $\ke$.

\begin{theorem}\label{thmA}
	Let the parameters $\kappa \in\mathbb{R}$ and $0\ne c \in\mathbb{C}$  be constrained such that $\kappa$ is not a nonnegative integer and \begin{equation}\label{st8}
		\frac{2\kappa}{e}(4\sin(1)+3-e)\geq(e+1)|c|+4(e-1)^3+6(e-1)^2+\frac{6}{e}-\frac{12}{e^2}
		\end{equation} then $6\kappa(1-\mathfrak{u}_{\nu})/c\in\ke$.
\end{theorem}
\begin{proof}
Let $\chi_v(z)=6\kappa(1-\mathfrak{u}_{\nu}(z))/c$ and define a function $p\colon\mathbb{D}\to\mathbb{C}$  by \[p(z)=1+\frac{z\chi''_{\nu}(z)}{\chi'_{\nu}(z)}=1+\frac{z\mathfrak{u}''_{\nu}(z)}{\mathfrak{u}'_{\nu}(z)}.  \] Using \eqref{st8}, we have  $\kappa >|c|/4$ so that $|\mathfrak{u}'_{\nu}(z)|>0$ by means of  \cite[Equation 2.19, p.~10]{MR3718596}  and therefore the function $p$
is analytic in $\mathbb{D}$ with $p(0)=1$. The function $\mathfrak{u}_{\nu}$ satisfies the differential equation given by \eqref{st5} and by twice  differentiating  \eqref{st5}, we obtain \[4z^2\mathfrak{u}^{(4)}_{\nu}(z)+2(2\kappa+9)z\mathfrak{u}'''_{\nu}(z)+(10(\kappa+1)+cz)\mathfrak{u}''_{\nu}(z)+2c\mathfrak{u}'_{\nu}(z)=0. \]
Assume that $z\ne0$. Since $\mathfrak{u}'_{\nu}(z)\not=0$ for all $z\in\mathbb{D}$,  dividing the above equation by $\mathfrak{u}'_{\nu}$ and then multiplying it by $z$, we have
  \[ 4 \frac{z^3\mathfrak{u}^{(4)}_{\nu}(z)}{\mathfrak{u}'_{\nu}(z)}+2(2\kappa+9) \frac{z^2\mathfrak{u}'''_{\nu}(z)}{\mathfrak{u}'_{\nu}(z)}+10(\kappa+1)\frac{z\mathfrak{u}''_{\nu}(z)}{\mathfrak{u}'_{\nu}(z)} +cz\frac{z\mathfrak{u}''_{\nu}(z)}{\mathfrak{u}'_{{\nu}}(z)}+2cz=0. \]
 Differentiating the equation $z\mathfrak{u}''_\nu(z)/\mathfrak{u}'_\nu(z)=p(z)-1$ logarithmically implies that the function $p$ satisfies the following differential equation
  \begin{gather*}
4z^2 p''(z)+4zp'(z)+2(2\kappa-3)zp'(z)+12zp(z)p'(z)-2\kappa (p(z)-1)\\+\;4\kappa(p^2(z)-1)+4(p(z)-1)^3+6(p(z)-1)^2+c(p(z)+1)z=0.
\end{gather*}
 Observe that the above equation is also true for $z=0$. Define a function $\Psi(r,s,t;z):=4(t+s)+2(2\kappa-3)s+12rs-2\kappa(r-1)+4\kappa(r^2-1)+4(r-1)^3+6(r-1)^2+c(r+1)z$ and let $\Omega:=\{0\}$. As done in previous results, a straightforward calculation shows that
 \begin{align*}
|\Psi(r,s,t;z)|&=\big| 2e^{e^{i\theta}}\big(2(t+s)e^{-e^{i\theta}}+(2\kappa-3)me^{i\theta}+6me^{i\theta}e^{e^{i\theta}}-\kappa(1-e^{-e^{i\theta}})\\ & \quad +\; 2\kappa(e^{e^{i\theta}}-e^{-e^{i\theta}})\big)+4 (e^{e^{i\theta}}-1)^3+6(e^{e^{i\theta}}-1)^2+c(e^{e^{i\theta}}+1)z\big|\\
& \geq 2e^{\cos\theta} \big|2(t+s)e^{-i\theta}e^{-e^{i\theta}}+(2\kappa-3)m+6me^{e^{i\theta}}-\kappa(1-e^{-e^{i\theta}})e^{-i\theta}\\
 & \quad +\; 2\kappa(e^{e^{i\theta}}-e^{-e^{i\theta}})e^{-i\theta}\big|-4 |e^{e^{i\theta}}-1|^3-6|e^{e^{i\theta}}-1|^2-|c|\cdot|e^{e^{i\theta}}+1|\cdot|z|\\
& > \frac{2}{e}\big[2\re((t+s)e^{-i\theta}e^{-e^{i\theta}})+(2\kappa-3)m+6m \re(e^{e^{i\theta}})-\kappa\re((1-e^{-e^{i\theta}})e^{-i\theta})\\
&\quad+\; 2\kappa\re((e^{e^{i\theta}}-e^{-e^{i\theta}})e^{-i\theta})\big]-4(e-1)^3-6(e-1)^2-(e+1)|c|\\
& \geq \frac{2}{e}\big[2\kappa-3+\frac{6}{e}-\kappa\re((1-e^{-e^{i\theta}})e^{-i\theta})+2\kappa\re((e^{e^{i\theta}}-e^{-e^{i\theta}})e^{-i\theta})\big]\\
&\quad-\; 4(e-1)^3-6(e-1)^2-(e+1)|c|
\end{align*}
whenever $r = e^{e^{i\theta}}$, $s = m e^{i\theta} e^{e^{i\theta}}$ and $\re ((s+t)e^{-i\theta}e^{-e^{i\theta}}) \geq 0 $  for $z\in\mathbb{D}$,  $\theta \in [0,2\pi)$ and $m \geq 1$. Here we have used the fact that $\kappa\geq3/2$ which follows by \eqref{st8}. It remains to find the maximum and minimum values of the functions $\re((1-e^{-e^{i\theta}})e^{-i\theta})$ and $\re((e^{e^{i\theta}}-e^{-e^{i\theta}})e^{-i\theta})$ respectively. The minimum value of the latter has already been determined in Theorem \ref{lom4}. For the maximum value of the former expression, let $l(\theta):=\re((1-e^{-e^{i\theta}})e^{-i\theta})=\cos\theta - e^{-\cos\theta} \cos(\theta+\sin\theta)$. Then $l''(0)=3/e-1>0$ and $l''(\pi)=1-e<0$ at critical points of the function $l$. Hence by the second derivative test, we get \[\max_{\theta\in[0,2\pi)}l(\theta)=l(\pi)=e-1. \]
Thus
\begin{align*}
|\Psi(r,s,t;z)|& > \frac{2}{e}\left[2\kappa-3+\frac{6}{e} -\kappa(e-1)+ 4\kappa\sin(1)\right] -4(e-1)^3-6(e-1)^2-(e+1)|c|\\
& = \frac{2\kappa}{e}(4\sin(1)+3-e)-\frac{6}{e}+\frac{12}{e^2}-4(e-1)^3-6(e-1)^2-(e+1)|c|\geq0
\end{align*}
using the assumption  \eqref{st8}. Therefore $p(z)\prec e^z$ for all $z\in\mathbb{D}$ by Lemma \ref{lemA} and hence the function $\chi_{\nu}$ is exponential convex.
\end{proof}


Choosing $b=1$ and $c=1$, we obtain  the Struve function  given by \eqref{st9} which satisfies the differential equation \eqref{st1}. In particular, using Theorem \ref{thmA}, we have
\begin{corollary}\label{cora}
	Let the parameter $\nu\in\mathbb{R}$ be such that $\nu+3/2$ is not a nonnegative integer and the function $\mathcal{H}_\nu\colon\mathbb{D}\to\mathbb{C}$ be defined by \begin{equation*}\label{be}\mathcal{H}_{\nu}(z)=2^\nu\sqrt{\pi} \Gamma\left(\nu+\frac{3}{2}\right)z^{-(\nu+1)} \emph{\textbf{H}}_\nu(z)\end{equation*}
	where $\emph{\textbf{H}}_\nu$ is the Struve function of the first kind of order $\nu$ defined in \eqref{st9}. If \begin{equation}\label{st12}
	\nu\geq\frac{e}{8\sin(1)+6-2e}\left[4(e-1)^3+6(e-1)^2+(e+1)+\frac{6}{e}-\frac{12}{e^2}\right]-\frac{3}{2}
	\end{equation}
 then the function $-3(2\nu+3)(\mathcal{H}_\nu-1)\in\ke$ and $-3(2\nu+3)z\mathcal{H}'_\nu\in\se$.
\end{corollary}
Observe that if we choose $b=1$ and $c=-1$ in  Theorem \ref{thmA}, then the function $\mathcal{L}_\nu$ defined by \begin{equation}\label{st11}\mathcal{L}_{\nu}(z)=2^\nu\sqrt{\pi} \Gamma\left(\nu+\frac{3}{2}\right)z^{-(\nu+1)} \textbf{L}_\nu(z)\end{equation}
where $\textbf{L}_\nu$ is the modified Struve function of the first kind of order $\nu$ defined by \eqref{st10}, has the same properties as that of the function $\mathcal{H}_\nu$ because of the fact that $|c|=1$.
More precisely, we have
\begin{corollary}
	Let the parameter $\nu\in\mathbb{R}$ and the function $\mathcal{L}_\nu\colon\mathbb{D}\to\mathbb{C}$ be defined by \eqref{st11}. If the condition \eqref{st12} holds,
	then the function $3(2\nu+3)(\mathcal{L}_\nu-1)\in\ke$ and $3(2\nu+3)z\mathcal{L}'_\nu\in\se$.
\end{corollary}

Let   $f$ and $g$ be two analytic functions having the power series expansion as  $f(z)=z+\sum_{n=1}^\infty a_{n+1}z^{n+1}$ and $g(z)=z+\sum_{n=1}^\infty b_{n+1}z^{n+1}$ respectively. The Hadamard product  $f\ast g$ (or convolution) of $f$ and $g$ is  defined as the power series  $(f\ast g)(z)=z+\sum_{n=1}^\infty a_{n+1}b_{n+1}z^{n+1}$. Note that  both the classes $\ke$ and $\se$ are closed under the convolution with convex functions \cite[Theorem 2.9, p.~377]{MR3394060}. In general, the Alexander operator  $A:\mathscr{A}\to\mathscr{A}$  for the function $f$ is defined   by \begin{equation}\label{alx} A[f](z):=\int_{0}^{z} \frac{f(t)}{t} dt=-\log(1-z)\ast f(z) \qquad (z\in\mathbb{D}) \end{equation} and the Libera operator $L:\mathscr{A}\to\mathscr{A}$ is given by  \[ L[f](z):=\frac{2}{z}\int_{0}^{z} f(t) dt=\frac{-2(z+\log(1-z))}{z}\ast f(z) \qquad (z\in\mathbb{D}). \]
The following theorem provides a property of the Alexander and Libera transforms for the generalized Struve function.
\begin{theorem}\label{thmC}
	If the parameters $\kappa$ and $c$ are constrained as in Theorem \ref{thmA}, then $6\kappa(1-\mathfrak{u}_{\nu})/c \ast f \in\ke$ for every $f\in\mathscr{K}$ and therefore the functions $A\left[6\kappa(1-\mathfrak{u}_{\nu})/c\right]$ and $L\left[6\kappa(1-\mathfrak{u}_{\nu})/c\right]$ belong to the class $\ke$.
\end{theorem}

\end{document}